\documentclass{elsarticle}
\usepackage{amsmath}
\usepackage{amssymb}
\usepackage{comment}
\usepackage{enumitem}
\usepackage{fullpage}

\begin{document}

\newtheorem{theorem}{Theorem}[section]
\newtheorem{thm*}{Theorem}[section]
\newtheorem{corollary}[theorem]{Corollary}
\newtheorem{proposition}[theorem]{Proposition}
\newtheorem{lemma}[theorem]{Lemma}
\newtheorem{sublemma}[theorem]{Sublemma}
\newtheorem{lemma*}{Lemma}
\newtheorem{conjecture}[theorem]{Conjecture}

\newdefinition{definition}[theorem]{Definition}
\newdefinition{remark}[theorem]{Remark}
\newdefinition{question}[theorem] {Question}

\newproof{proof}{Proof}


\newdefinition{example}[theorem]{Example}
\newtheorem{notation}[theorem]{Notation}
\newtheorem{terminology}[theorem]{Terminology}
\newtheorem{question*}{Question}

\numberwithin{equation}{theorem}

\setcounter{section}{0}

\date\today

\title{Rational points and orbits on the variety of elementary subalgebras}

\author[USC]{Jared Warner\corref{cor1}}
\ead{hjwarner@usc.edu}

\cortext[cor1]{Corresponding author: phone number 1-805-698-4463}


\address[USC]{USC Dornsife, Department of Mathematics, 3620 S. Vermont Ave., KAP 104,
Los Angeles, CA 90089-2532, USA}

\begin{abstract}
For $G$ a connected, reductive group over an algebraically closed field $k$ of large characteristic, we use the canonical Springer isomorphism between the nilpotent variety of $\mathfrak{g}:=\mathrm{Lie}(G)$ and the unipotent variety of $G$ to study the projective variety of elementary subalgebras of $\mathfrak{g}$ of rank $r$, denoted $\mathbb{E}(r,\mathfrak{g})$.  In the case that $G$ is defined over $\mathbb{F}_p$, we define the category of $\mathbb{F}_q$-expressible subalgebras of $\mathfrak{g}$ for $q=p^d$, and prove that this category is isomorphic to a subcategory of Quillen's category of elementary abelian subgroups of the finite Chevalley group $G(\mathbb{F}_q)$.  This isomorphism of categories leads to a correspondence between $G$-orbits of $\mathbb{E}(r,\mathfrak{g})$ defined over $\mathbb{F}_q$ and $G$-conjugacy classes of certain elementary abelian subgroups of rank $rd$ in $G(\mathbb{F}_q)$ which satisfy a closure property characterized by the Springer isomorphism.  We use Magma to compute examples for $G=\mathrm{GL}_n$, $n\le 5$.
\end{abstract}

\maketitle

In \cite{CFP}, J. Carlson, E. Friedlander, and J. Pevtsova initiated the study of $\mathbb{E}(r,\mathfrak{g})$, the projective variety of rank $r$ elementary subalgebras of a restricted lie algebra $\mathfrak{g}$.  The authors demonstrate that the study of $\mathbb{E}(r,\mathfrak{g})$ informs the representation theory and cohomology of $\mathfrak{g}$.  This is all reminiscent of the case of a finite group $G$, where the elementary abelian $p$-subgroups play a significant role in the story of the representation theory and cohomology of $G$, as first explored by Quillen in \cite{Quillen1} and \cite{Quillen2}.

In this paper, we further explore the structure of $\mathbb{E}(r,\mathfrak{g})$ and its relationship with elementary abelian subgroups.  Theorem \ref{IsoCat} shows in the case that $\mathfrak{g}$ is the Lie algebra of a connected, reductive group $G$ defined over $\mathbb{F}_p$, the category of $\mathbb{F}_q$-expressible subalgebras (Definitions \ref{FpExpressible} and \ref{CatOfSubAlg}) is isomorphic to a subcategory of Quillen's category of elementary abelian $p$-subgroups of $G(\mathbb{F}_q)$, where $q=p^d$.  Specifically, we introduce the notion of an $\mathbb{F}_q$-linear subgroup (Definition \ref{FqLinear}), and we show in Corollary \ref{Bij} that the $\mathbb{F}_q$-expressible subalgebras of rank $r$ are in bijection with the $\mathbb{F}_q$-linear elementary abelian subgroups of rank $rd$ in $G(\mathbb{F}_q)$.  This bijection leads to Corollary \ref{MaximalElemAb}, which allows us to compute the largest integer $R=R(\mathfrak{g})$ such that $\mathbb{E}(R,\mathfrak{g})$ is nonempty for a simple Lie algebra $\mathfrak{g}$.  These values are presented in Table \ref{Largest}.

The results and definitions in \S \ref{AnIsoOfCat} rely on the canonical Springer isomorphism $\sigma:\mathcal{N}(\mathfrak{g})\to\mathcal{U}(G)$, which has been shown to exist under the hypotheses we assume in this paper, as detailed in \cite{Seitz}, \cite{CLN}, \cite{Sobaje1}, and \cite{McNinch}. Together with Lang's theorem, Theorem \ref{IsoCat} implies Theorem \ref{FinOrb}, which establishes a natural bijection between the $G$-orbits of $\mathbb{E}(r,\mathfrak{g})$ defined over $\mathbb{F}_q$ and the $G$-conjugacy classes of $\mathbb{F}_q$-linear elementary abelian subgroups of rank $rd$ in $G(\mathbb{F}_q)$.  Example \ref{InfOrb}, due to R. Guralnick, shows that $\mathbb{E}(r,\mathfrak{g})$ may be an infinite union of $G$-orbits (in fact this is usually the case).  However, Proposition \ref{FriedConj} demonstrates that $\mathbb{E}(R(\mathfrak{g}),\mathfrak{g})$ is a finite union of orbits for all connected, reductive $G$ such that $(G,G)$ is an almost-direct product of simple groups of classical type.  We believe that $\mathbb{E}(R(\mathfrak{g}),\mathfrak{g})$ is a finite union of orbits for all connected, reductive groups, and Proposition \ref{FriedConj} reduces the verification of this belief to proving a claim about conjugacy classes of elementary abelian $p$-subgroups in $G(\mathbb{F}_q)$ for varying $d$ and for exceptional simple groups $G$.  Our interest in describing the $G$-orbits is motivated by the results of \S 6 in \cite{CFP}, where the authors construct algebraic vector bundles on $G$-orbits of $\mathbb{E}(r,\mathfrak{g})$ associated to a rational $G$-module $M$ via the restriction of image, cokernel, and kernel sheaves.

Through personal communication with the author, E. Friedlander asked for conditions implying that $\mathbb{E}(r,\mathfrak{g})$ is irreducible.  In the case that $\mathfrak{g}=\mathfrak{gl}_n$, Theorem \ref{Irreducible} presents certain ordered pairs $(r,n)$ for which $\mathbb{E}(r,\mathfrak{g})$ is irreducible.  This theorem relies on previous results concerning the irreducibility of $\mathcal{C}_r(\mathcal{N}(\mathfrak{gl}_n))$, the variety of $r$-tuples of pair-wise commuting, nilpotent $n\times n$ matrices (see \cite{Sivic} for a nice summary of these results).

Finally, in \S \ref{Group}, we compute a few examples for $G=\mathrm{GL}_n$.  Some of the computations depend on Conjecture \ref{OrbitSize}, which supposes the dimension of an orbit is related to the size of the corresponding $G$-conjugacy class.  Equation \eqref{Dim} computes the dimension of $\mathbb{E}(r,\mathfrak{gl}_n)$ for all $(r,n)$ such that $\mathcal{C}_r(\mathcal{N}(\mathfrak{gl}_n))$ is irreducible, and surprisingly this equation agrees with computations of $\dim(\mathbb{E}(r,\mathfrak{gl}_n))$ even for ordered pairs where $\mathcal{C}_r(\mathcal{N}(\mathfrak{gl}_n))$ is known to be reducible.  Proposition \ref{RegOrb} computes the dimension of the open orbit defined by a regular nilpotent element, as first considered in Proposition $3.19$ of \cite{CFP}.  For $n\le 5$, we bound the number of $G$-orbits in $\mathbb{E}(r,\mathfrak{gl}_n)$ defined over $\mathbb{F}_q$ and compute their dimensions. 

\section{Review and Preliminaries}\label{Review}

Let $k$ be an algebraically closed field of characteristic $p>0$, and let $G$ be a connected, reductive algebraic group over $k$, with coxeter number $h=h(G)$.  Following \S 2 in \cite{Steinberg}, we let $\pi=\pi(G)$ denote the fundamental group of $G'=(G,G)$.  We will often require that $p$ satisfies the following two conditions, which will be collectively referred to as condition $(\star)$:
\begin{equation}
\tag{$\star$}
\begin{aligned}
(1)&\;p\ge h\\
(2)&\;p \nmid |\pi|\\
\end{aligned}
\end{equation}
We make three remarks about condition $(\star)$.  First, $(1)$ implies $(2)$ in all cases except when $p=h$ and $G'$ has an adjoint component of type $A$.  Second, $(2)$ is equivalent to the separability of the universal cover $G'_{sc}\to G'$ (\cite{Steinberg},\S 2.4).  For example, the canonical map $\operatorname{SL}_p\to\operatorname{PSL}_p$ is not separable in characteristic $p$, so we must exclude the case $G=\operatorname{PSL}_p$.  Third, $(\star)$ implies that $p$ is non-torsion for $G$ (cf. \S 2 in \cite{LMT}), which we require to use Theorem 2.2 of \cite{LMT} in our proof of Theorem \ref{CanonSpringer}.

The unipotent elements of $G$ form an irreducible closed subvariety of $G$, denoted $\mathcal{U}(G)$, and $G$ acts by conjugation on $\mathcal{U}(G)$.  In the Lie algebra setting, the nilpotent elements of $\mathfrak{g}:=\mathrm{Lie}(G)$ also form an irreducible closed subvariety of $\mathfrak{g}$, denoted $\mathcal{N}(\mathfrak{g})$, and $\mathcal{N}(\mathfrak{g})$ is a $G$-variety under the adjoint action of $G$ on $\mathfrak{g}$.  The main tool we will use to translate information between the group and Lie algebra settings will be a well-behaved Springer isomorphism.

\begin{definition}
A \emph{Springer isomorphism} is a $G$-equivariant isomorphism of algebraic varieties $\sigma:\mathcal{N}(\mathfrak{g})\to\mathcal{U}(G)$.
\end{definition}

\noindent In \cite{Springer}, Springer shows that Springer isomorphisms exist when $p$ is very good, but a note of Serre in (\cite{McNinch}, \S 10) mentions that in general they are not unique.  In fact, they are parametrized by a variety of dimension equal to $\mathrm{rank}(G)$.

\begin{example}[\cite{Sobaje}, \S 3]\label{TruncExp}
Let $G=SL_n$.  Then the Springer isomorphisms are parameterized by the variety $a_1\ne 0$ in $\mathbb{A}^{n-1}$ by $\sigma_{(a_1,\ldots,a_{n-1})}(X)=1+a_1X+\ldots+a_{n-1}X^{n-1}$ where $X^n=0$.  It follows that different Springer isomorphisms can behave very differently.  Here, since $p\ge h(SL_n)=n$, we have the particularly nice choice of Springer isomorphism 
\[
\sigma(X)=1+X+\frac{X^2}{2!}+\ldots+\frac{X^{p-1}}{(p-1)!}
\]
This is just the truncated exponential series, which we will denote $\exp$.
\end{example}

The truncated exponential series considered in Example \ref{TruncExp} has the following convenient property for $p\ge n$:
\begin{equation}\label{SumToProd}
[X,Y]=0\quad\Longrightarrow\quad\exp(X+Y)=\exp(X)\exp(Y)
\end{equation}
We give a brief proof of \eqref{SumToProd}.  Serre records in (\cite{Serre}, (4.1.7)) that $\exp(X)\exp(Y)=\exp(X+Y-W_p(X,Y))$ for two commuting elements $X,Y$ and arbitrary $p$, where $W_p(X,Y)=\frac{1}{p}((X+Y)^p-X^p-Y^p)$.  For $p\ge n$, we have $X^p=Y^p=(X+Y)^p=0$ so that $W_p(X,Y)=0$ and we recover \eqref{SumToProd}.

Proposition $5.3$ in \cite{Seitz} states for any parabolic subgroup $P$ in $G$ whose unipotent radical $U_P$ has nilpotence class less than $p$, there is a unique $P$-equivariant isomorphism $\varepsilon_P:\mathrm{Lie}(U_P)=\mathfrak{u}_P\to U_P$ satisfying the following conditions.

\begin{enumerate}
\item $\varepsilon_P$ is an isomorphism of algebraic groups, where $\mathfrak{u}_P$ has the structure of an algebraic group via the Baker-Campbell-Hausdorff formula (notice the condition on the nilpotence class of $\mathfrak{u}_P$ is required for this group law to make sense).
\item The differential of $\varepsilon_P$ is the identity on $\mathfrak{u}_P$.
\end{enumerate}

In Theorem 3 of \cite{CLN}, the authors uniquely extend this isomorphism on $\mathfrak{u}_P$ to all of $\mathcal{N}(\mathfrak{g})$ for $G$ simple, with weaker conditions on $p$ than we consider in this paper.  Specifically, they require that $p$ is good, that $\mathcal{N}(\mathfrak{g})$ is normal and that $G_{sc}\to G$ is separable.  Condition $(1)$ of $(\star)$ implies that $p$ is good and that $\mathcal{N}(\mathfrak{g})$ is normal, and we've already noted that condition $(2)$ is equivalent to the separability of $G_{sc}\to G$.  We now show that under our assumptions on $p$, the result of Theorem $3$ in \cite{CLN} may be extended to reductive groups, and the canonical isomorphism obtained sends sums to products much like the truncated exponential (cf \eqref{SumToProd}).

\begin{theorem}\label{CanonSpringer}
For $G$ a connected, reductive algebraic group, and for $p$ satisfying condition $(\star)$, there is a (necessarily) unique Springer isomorphism $\sigma:\mathcal{N}(\mathfrak{g})\to\mathcal{U}(G)$ which restricts to the canonical isomorphism of \cite{Seitz} on all $\mathfrak{u}_P$ for $P$ any parabolic subgroup of $G$.  This Springer isomorphism has the following properties:
\begin{enumerate}
\item $[X,Y]=0$ if and only if $(\sigma(X),\sigma(Y))=e$.
\item If $G$ is defined over $\mathbb{F}_p$, then so is $\sigma$.
\item If $[X,Y]=0$, then $\sigma(X+Y)=\sigma(X)\sigma(Y)$.
\end{enumerate}
\end{theorem}

\begin{proof}
First, we note that $\mathcal{U}(G)\subset G'=(G,G)$ and $\mathcal{N}(\mathfrak{g})\subset [\mathfrak{g},\mathfrak{g}]\subset \operatorname{Lie}(G')$, so that an isomorphism of the nilpotent and unipotent varieties of $G'$ is also one for $G$.  Furthermore, since $G$ is the product of $Z(G)$ and $G'$, a $G'$-equivariant map $\mathcal{N}(\mathfrak{g})\to\mathcal{U}(G)$ is also $G$-equivariant map.  Hence, we may assume that $G=G'$ is semisimple.

If $G$ is semisimple, there is an isogeny $H\to G$ where $H$ is a product of simple groups.  By our assumptions on $p$, $H\to G$ is separable, so that the induced map on Lie algebras is an isomorphism.  Hence we may assume that $G$ is a product of simple groups.

For simple groups $G$ (and hence products of simple groups) Theorem 3 of \cite{CLN} states that there is a unique Springer isomorphism with properties $1$ and $2$ which restricts to the canonical isomorphism of \cite{Seitz} for all parabolic subgroups $P$ whose unipotent radical has nilpotence class less than $p$.  Since $p\ge h$, all parabolic subgroups satisfy this criteria.

Finally, to see that $\sigma$ has property $3$, suppose $X$ and $Y$ are commuting nilpotent elements.  Theorem $2.2$ in \cite{LMT} states that there is some Borel subgroup $B\subset G$ with unipotent radical $U$ such that $X,Y\in\mathrm{Lie}(U).$  Since $\sigma$ restricts to the canonical isomorphism on $\mathfrak{u}_B$, it follows that $\sigma(X * Y)=\sigma(X)\sigma(Y)$, where $*$ is the group operation defined by the Baker-Campbell-Hausdorff formula, which for commuting elements satisfies $X*Y=X+Y$.  Property $3$ follows.
\end{proof}

We will use the canonical Springer isomorphism $\sigma$ to study the projective variety $\mathbb{E}(r,\mathfrak{g})$, as defined in \cite{CFP}.  The following discussion is relevant for an arbitrary restricted Lie algebra $(\mathfrak{g},[\cdot,\cdot],(\cdot)^{[p]})$, but we are only concerned with the case $\mathfrak{g}=\mathrm{Lie}(G)$.

\begin{definition}[\cite{CFP}, Definition 1.2]
An \emph{elementary subalgebra} $\epsilon\subset\mathfrak{g}$ is an abelian Lie subalgebra of $\mathfrak{g}$ with trivial restriction, i.e., $x^{[p]}=0$ for all $x\in \epsilon$.
\end{definition}

Let $\mathbb{E}(r,\mathfrak{g})$ be the set of elementary subalgebras of rank $r$ in $\mathfrak{g}$.  Considering $\epsilon\hookrightarrow\mathfrak{g}$ as an inclusion of vector spaces, there is an embedding $\mathbb{E}(r,\mathfrak{g})\hookrightarrow \operatorname{Grass}(r,\mathfrak{g})$.  This is a closed embedding so that $\mathbb{E}(r,\mathfrak{g})$ has the structure of a projective subvariety of $\operatorname{Grass}(r,\mathfrak{g})$ (\cite{CFP}, Proposition 1.3).  If $\mathfrak{g}$ is the Lie algebra of an algebraic group $G$, then $\mathbb{E}(r,\mathfrak{g})$ is a $G$-variety via the adjoint action of $G$ on $\mathfrak{g}$.  Specifically, for any $\epsilon\in\mathbb{E}(r,\mathfrak{g})$ and any $g\in G$, the image of $\epsilon$ under $\mathrm{Ad}_g:\mathfrak{g}\to\mathfrak{g}$ is elementary of rank $r$.

We note for later purposes the following construction, which appears in Proposition 1.3 of \cite{CFP} and its proof.  Let $\mathcal{C}_r(\mathcal{N}(\mathfrak{g}))^{\circ}$ denote the variety of $r$-tuples of pairwise-commuting, nilpotent, linearly independent elements of $\mathfrak{g}$.  By taking the $k$-span of elements in an $r$-tuple, any $(X_1,\ldots,X_r)\in\mathcal{C}_r(\mathcal{N}(\mathfrak{g}))^{\circ}$ defines an elementary subalgebra of rank $r$, so there is a map of algebraic varieties $\mathcal{C}_r(\mathcal{N}(\mathfrak{g}))^{\circ}\twoheadrightarrow \mathbb{E}(r,\mathfrak{g})$.

\section{$\mathbb{F}_q$-expressability and $\mathbb{F}_q$-rational points of $\mathbb{E}(r,\mathfrak{g})$}\label{Rationality}

The following definitions are motivated by (\cite{CLN},\S 3).  In all that follows we suppose that $G$ has a fixed $\mathbb{F}_p$-structure, i.e. $G=G_0\times_{\mathbb{F}_p}\mathrm{Spec}\;k$ for some fixed algebraic group $G_0$ over $\mathbb{F}_p$.  It follows that $\mathfrak{g}:=\text{Lie}(G)$ has an $\mathbb{F}_p$-structure coming from $\mathfrak{g}_0:=\mathrm{Lie}(G_0)$ given by $\mathfrak{g}=\mathfrak{g}_0\times_{\mathbb{F}_p}\mathrm{Spec}\;k$.  For $q=p^d$, by abuse of notation we write $G(\mathbb{F}_q)$ (resp. $\mathfrak{g}(\mathbb{F}_q))$ to denote the $\mathbb{F}_q$-rational points of $G_0$ (resp. $\mathfrak{g}_0)$.  For another point of view, we may consider $G(\mathbb{F}_q)$ to be the subgroup consisting of all $k$-points of $G$ obtained from the base-change of an $\mathbb{F}_q$-point of $G_0$ (and similarly for the Lie algebra).  Here, we view the vector space $\mathfrak{g}$ as a scheme over $k$ via the following standard construction.  Given a finite dimensional vector space $V$ over $k$, give $V$ the structure of a linear scheme over $k$ with coordinate algebra $S^*(V^{\#})$.  Then the $k$-points of $V$ with this scheme structure are naturally identified with the elements of the vector space $V$.  In particular, in the setting just discussed, $\mathfrak{g}(\mathbb{F}_q)\cong \mathfrak{g}_0$.

\begin{definition}[\cite{CLN},\S 3, Definition 1]
An element $X\in\mathfrak{g}$ is \emph{$\mathbb{F}_q$-expressible} if it can be written as $X=\sum c_iX_i$ with $c_i\in k$, $X_i\in\mathfrak{g}(\mathbb{F}_q)$, $X_i^{[p]}=0=[X_i,X_j].$
\end{definition}

\noindent This definition can be extended to the notion of an $\mathbb{F}_q$-expressible subalgebra.

\begin{definition}\label{FpExpressible}
We call an elementary subalgebra $\epsilon$ an \emph{$\mathbb{F}_q$-expressible subalgebra}, if it has a basis of the form $\{X_1,\ldots,X_r\}\subset\mathfrak{g}(\mathbb{F}_q)$, i.e., $\epsilon=\epsilon(\mathbb{F}_q)\otimes_{\mathbb{F}_q}k$.
\end{definition}

\noindent To speak of $\mathbb{F}_q$-rational points of $\mathbb{E}(r,\mathfrak{g})$, we require a rationality condition on $\mathfrak{g}$.  Notice that if $\mathfrak{g}$ has an $\mathbb{F}_p$-structure given by $\mathfrak{g}=\mathfrak{g}_0\otimes_{\mathbb{F}_p}k$, then $\mathbb{E}(r,\mathfrak{g})$ is defined over $\mathbb{F}_p$.  This can be seen as follows.  Fix an embedding $\mathfrak{g}_0\hookrightarrow\mathfrak{gl}_n(\mathbb{F}_p)$.  Then $\mathfrak{g}\hookrightarrow\mathfrak{gl}_n$ is determined by linear equations with coefficients in $\mathbb{F}_p$.  The equations defining the nilpotent, commuting, and linearly independent conditions are all homogeneous polynomials with coefficients in $\mathbb{F}_p$ as well.

We claim that the $\mathbb{F}_q$-rational points of $\mathbb{E}(r,\mathfrak{g})$ are precisely the $\mathbb{F}_q$-expressible subalgebras of $\mathfrak{g}$.  To see this, notice that $\mathbb{E}(r,\mathfrak{g})(\mathbb{F}_q)=\mathbb{E}(r,\mathfrak{g})\cap\mathrm{Grass}(r,\mathfrak{g})(\mathbb{F}_q)$.  The claim follows from the fact that the $\mathbb{F}_q$-rational points of the Grassmannian are those $r$-planes with a basis in $\mathfrak{g}(\mathbb{F}_q)$.

\section{The Category of $\mathbb{F}_q$-expressible Elementary Subalgebras}\label{AnIsoOfCat}

We begin by recalling the category of elementary abelian $p$-subgroups of a finite group, first considered by Quillen.  For $g\in G$, let $c_g:G\to G$ be defined by $c_g(h)=ghg^{-1}$

\begin{definition}\label{CatOfSubGrp}
Let $\Gamma$ be a finite group, and let $p$ be a prime dividing the order of $\Gamma$.  Define $\mathcal{C}_{\Gamma}$ to be the category whose objects are the elementary abelian $p$-subgroups of $\Gamma$, and whose morphisms are group homomorphisms $E\to E'$ which can be written as the composition of an inclusion followed by $c_g$ for some $g\in \Gamma$.  In particular we have a morphism $E\to E'$ if and only if $E$ is conjugate to a subgroup of $E'$.
\end{definition}

Motivated by Definition \ref{CatOfSubGrp}, we make similar definitions for restricted Lie algebras.

\begin{definition}\label{CatOfSubAlg}
\leavevmode
\begin{enumerate}
\item Let $G$ be an algebraic group defined over $k$, and let $\mathfrak{g}=\operatorname{Lie}(G)$.  Define $\mathcal{C}_{\mathfrak{g}}$ to be the category whose objects are the elementary abelian subalgebras of $\mathfrak{g}$, and whose morphisms are inclusions followed by $\operatorname{Ad}_g$ for $g\in G$.
\item Inside of $\mathcal{C}_{\mathfrak{g}}$, let $\mathcal{C}_{\mathfrak{g}}(\mathbb{F}_q)$ be the subcategory whose objects are $\mathbb{F}_q$-expressible subalgebras, and whose morphisms are inclusions composed with $\operatorname{Ad}_g$ for $g\in G(\mathbb{F}_q)$.
\end{enumerate}
\end{definition}

The following theorem further emphasizes that elementary subalgebras of a restricted Lie algebra are the appropriate cohomological and representation theoretic analog to elementary abelian subgroups of a finite group.  The existence of the canonical Springer isomorphism is used in the proofs of this section, so throughout we assume $p$ satisfies condition $(\star)$.

\begin{theorem}\label{IsoCat}
Let $G$ be a reductive, connected group.  Then for $q=p^d$, the category $\mathcal{C}_{\mathfrak{g}}(\mathbb{F}_q)$ is isomorphic to a subcategory of $\mathcal{C}_{G(\mathbb{F}_q)}$.  For $d=1$, Quillen's category $\mathcal{C}_{G(\mathbb{F}_p)}$ is isomorphic to $\mathcal{C}_{\mathfrak{g}}(\mathbb{F}_p)$.
\end{theorem}

\begin{proof}
Let $\sigma$ be the canonical Springer isomorphism from Theorem \ref{CanonSpringer}, and define a fully faithful functor $\mathcal{F}:\mathcal{C}_{\mathfrak{g}}(\mathbb{F}_q)\to\mathcal{C}_{G(\mathbb{F}_q)}$ as follows.
\[
\begin{aligned}
\mathcal{F}(\epsilon)=&\sigma(\epsilon(\mathbb{F}_q))\\
\mathcal{F}(\epsilon\hookrightarrow \epsilon')=&\sigma(\epsilon(\mathbb{F}_q))\hookrightarrow\sigma(\epsilon'(\mathbb{F}_q))\\
\mathcal{F}(\mathrm{Ad}_g)=&c_g
\end{aligned}
\]
where we note that if $\epsilon\subset\epsilon'$, then $\sigma(\epsilon(\mathbb{F}_q))\subset\sigma(\epsilon'(\mathbb{F}_q))$.  Since $\sigma$ is defined over $\mathbb{F}_p$, we have $\sigma(\epsilon(\mathbb{F}_q))\subset G(\mathbb{F}_q)$.  The fact that $\sigma(\epsilon(\mathbb{F}_q))$ is an elementary abelian $p$-group in $G(\mathbb{F}_q)$ follows from property $3$ of Theorem \ref{CanonSpringer}. Notice if $\epsilon$ has dimension $r$, then $|\epsilon(\mathbb{F}_q)|=q^r=p^{rd}$, so that $\sigma(\epsilon(\mathbb{F}_q))$ has rank $rd$.  That $\mathcal{F}$ is fully faithful follows from the definition of the morphisms in the categories $\mathcal{C}_{\mathfrak{g}}(\mathbb{F}_q)$ and $\mathcal{C}_{G(\mathbb{F}_q)}$.

In the case $d=1$, it remains to show that $\mathcal{F}$ is surjective on objects.  Let $E$ be any elementary abelian subgroup of rank $r$ in $G(\mathbb{F}_p)$.  I claim that $V=\sigma^{-1}(E)\subset\mathfrak{g}(\mathbb{F}_p)$ is an $r$-dimensional subspace over $\mathbb{F}_p$.  That $V$ is closed under addition follows from properties $1$ and $3$ of Theorem \ref{CanonSpringer}.  If $\lambda\in\mathbb{F}_p$, then $\lambda\sigma^{-1}(g)=\sigma^{-1}(g^{\lambda})$, so that $V$ is closed under $\mathbb{F}_p$-scalar multiplication.  Since $V$ has $p^r$ elements, it follows that it is an  $\mathbb{F}_p$-subspace of dimension $r$.  Then $\epsilon=V\otimes_{\mathbb{F}_p}k$ is an $r$-dimensional $k$-space such that $\mathcal{F}(\epsilon)=E$.
\end{proof}

Notice that in the proof of Theorem \ref{IsoCat}, the symbol $g^{\lambda}$ is meaningless for $\lambda\in\mathbb{F}_q\setminus\mathbb{F}_p$, which is why $\mathcal{F}$ is only an isomorphism for $d=1$.  The following example shows why $\mathcal{F}$ fails to be surjective for $d>1$.

\begin{example}\label{NotLinear}
Let $G=\mathrm{SL}_3$, let $d=2$, and let $\lambda\in\mathbb{F}_q\setminus\mathbb{F}_p$.  In this case, we have $\sigma (X)=I+X+\frac{1}{2}X^2$.  Consider the elementary abelian subgroup of rank $2$ defined as follows:
\[
E=\left\langle g=\begin{pmatrix}1&1&0\\0&1&0\\0&0&1\end{pmatrix},h=\begin{pmatrix}1&0&1\\0&1&0\\0&0&1\end{pmatrix}\right\rangle\subset G(\mathbb{F}_p)\subset G(\mathbb{F}_q)
\]
Then any subalgebra $\epsilon$ with $E\subset\mathcal{F}(\epsilon)$ must contain the elements $X=\lambda\sigma^{-1}(g)=\lambda(g-I)$ and $Y=\lambda\sigma^{-1}(h)=\lambda(h-I)$.  It follows that $I+X,I+Y\in\mathcal{F}(\epsilon)$, but $I+X,I+Y\notin E$, so that there is no subalgebra $\epsilon$ with $\mathcal{F}(\epsilon)=E$.
\end{example}

To determine the image of $\mathcal{F}$ for $d>1$, it will be helpful to define $g^{\lambda}:=\sigma(\lambda\sigma^{-1}(g))$ for $g\in \mathcal{U}(G)$ and $\lambda\in k$.  One can check using properties $1$ and $3$ of Theorem \ref{CanonSpringer} that the following familiar formulas hold for all $g,h,k\in \mathcal{U}(G)$ with $gh=hg$, and all $\lambda,\mu\in k$:

\begin{equation}\label{GenExp}
\begin{aligned}
g^{\lambda}\cdot g^{\mu}=&g^{\lambda+\mu}\\
(g^{\lambda})^{\mu}=&g^{\lambda\cdot\mu}\\
g^{\lambda}\cdot h^{\mu}=&h^{\mu}\cdot g^{\lambda}\\
(g\cdot k\cdot g^{-1})^{\lambda}=&g\cdot k^{\lambda}g^{-1}
\end{aligned}
\end{equation}

\begin{definition}\label{FqLinear}
Call an elementary abelian subgroup $E\subset G$ \emph{$\mathbb{F}_q$-linear} if $g^{\lambda}\in E$ for any $g\in E$ and any $\lambda\in\mathbb{F}_q$.
\end{definition}

Notice that any elementary abelian subgroup is $\mathbb{F}_p$-linear, as $g^{\lambda}$ is just given by the group operation in $E$ whenever $\lambda\in\mathbb{F}_p$.  Also notice that $E$ in Example \ref{NotLinear} is not $\mathbb{F}_q$-linear as $g^{\lambda}=I+X\notin E$.

\begin{lemma}\label{DivRank}
If $E\subset G$ is a finite $\mathbb{F}_q$-linear elementary abelian $p$-subgroup, then the rank of $E$ is divisible by $d$.
\end{lemma}

\begin{proof}
Choose any $g_1\in E$, and let $E_{g_1}=\{g_1^{\lambda}\mid\lambda\in\mathbb{F}_q\}$.  Then, $E_{g_1}\subset E$ by $\mathbb{F}_q$-linearity of $E$.  Also, using \eqref{GenExp} we see that $E_{g_1}$ is an $\mathbb{F}_q$-linear elementary abelian subgroup of rank $d$.  Choose $g_2\in E\setminus E_{g_1}$, and note that the $\mathbb{F}_q$-linearity of $E_{g_1}$ ensures that $E_{g_2}\cap E_{g_1}=e$.  Thus $E_{g_1}\times E_{g_2}\subset E$ is $\mathbb{F}_q$-linear of rank $2d$.  Since $E$ is finite, this process stops, and so there is a generating set $g_1,\ldots,g_r$ such that $E=E_{g_1}\times\ldots\times E_{g_r}$ has rank $rd$.
\end{proof}

\begin{remark}
The motivation behind the terminology ``$\mathbb{F}_q$-linear" can be seen in Lemma \ref{DivRank} and its proof.  Given any finite $\mathbb{F}_q$-linear elementary abelian subgroup $E$, there is a decomposition $E=E_{g_1}\times\ldots\times E_{g_r}$ for an appropriate choice of generators $g_1,\ldots,g_r$.  The rank of $E$ as an elementary abelian $p$-group is then $rd$.  Viewing $E_{g_i}$ as the ``span" of $g_i$, this decomposition is analogous to decomposing an $r$-dimensional $\mathbb{F}_q$-vector space into the direct sum of one dimensional subspaces spanned by vectors in a basis.
\end{remark}

\begin{proposition}\label{Subcat}
The image of $\mathcal{F}:\mathcal{C}_{\mathfrak{g}}(\mathbb{F}_q)\to\mathcal{C}_{G(\mathbb{F}_q)}$ is the full subcategory of $\mathbb{F}_q$-linear elementary abelian subgroups in $\mathcal{C}_{G(\mathbb{F}_q)}$.
\end{proposition}

\begin{proof}
First, let $\epsilon$ be $\mathbb{F}_q$-expressible, and choose any $g\in E=\sigma(\epsilon(\mathbb{F}_q))$.  Write $g=\sigma(X)$ for $X\in\epsilon(\mathbb{F}_q)$ and let $\lambda\in\mathbb{F}_q$.  Then $g^{\lambda}=\sigma(X)^{\lambda}=\sigma(\lambda X)\in E$ so that $E$ is $\mathbb{F}_q$-linear.

Conversely, suppose $E\subset G(\mathbb{F}_q)$ is $\mathbb{F}_q$-linear.  We proceed as in the proof of the $d=1$ case in Theorem \ref{IsoCat}.  The $\mathbb{F}_q$-linearity hypothesis on $E$ is precisely what is needed to show that $V=\sigma^{-1}(E)$ is closed under $\mathbb{F}_q$-scalar multiplication.  Notice that for $\lambda\in\mathbb{F}_q$ and for $\sigma^{-1}(g)\in V$, we have $\lambda\sigma^{-1}(g)=\sigma^{-1}(g^{\lambda})\in V$ because $g^{\lambda}\in E$.  It follows that the $\mathbb{F}_q$-expressible subalgebra $\epsilon=V\otimes_{\mathbb{F}_q}k$ satisfies $\mathcal{F}(\epsilon)=E$, which completes the proof.
\end{proof}

\begin{corollary}\label{Bij}
There is a bijection between $\mathbb{F}_q$-expressible elementary subalgebras of $\mathfrak{g}=\mathrm{Lie}(G)$ of rank $r$ and $\mathbb{F}_q$-linear elementary abelian subgroups of $G(\mathbb{F}_q)$ of rank $rd$.
\end{corollary}

\begin{proof}
The maps $\epsilon\mapsto\sigma(\epsilon(\mathbb{F}_q))$ and $E\mapsto\sigma^{-1}(E)\otimes_{\mathbb{F}_q}k$ used in the proof of Proposition \ref{Subcat} are inverse to each other.
\end{proof}

\begin{definition}
Let $R=R(\mathfrak{g})$ denote the largest integer such that $\mathbb{E}(R,\mathfrak{g})$ is nonempty.
\end{definition}

\begin{corollary}\label{MaximalElemAb}
Any elementary abelian subgroup $E\subset G(\mathbb{F}_q)$ is contained in an $\mathbb{F}_q$-linear elementary abelian subgroup of $G(\mathbb{F}_q)$.  In particular, any maximal elementary abelian subgroup is $\mathbb{F}_q$-linear.  Also, the largest rank of an elementary abelian subgroup of $G(\mathbb{F}_q)$ is $R(\mathrm{Lie}(G))d$.
\end{corollary}

\begin{proof}
To any $E\subset G(\mathbb{F}_q)$, consider $V=\langle \sigma^{-1}(E)\rangle,$ the $\mathbb{F}_q$-subspace of $\mathfrak{g}(\mathbb{F}_q)$ generated by $\sigma^{-1}(E)$.  Then $\sigma(V)=\mathcal{F}(V\otimes_{\mathbb{F}_q}k)$ is $\mathbb{F}_q$-linear, and $E\subset\sigma(V)$.

For the last statement of the corollary, let $\epsilon$ be an elementary subalgebra of rank $R=R(\mathrm{Lie}(G))$.  Then Corollary \ref{Bij} shows that $\sigma(\epsilon(\mathbb{F}_q))$ is elementary abelian of rank $Rd$.  If there exists $E\subset G(\mathbb{F}_q)$ of larger rank, then $E$ must lie in an $\mathbb{F}_q$-linear elementary abelian subgroup $E'$ of rank $R'd$ for $R'>R$.  Then $\sigma^{-1}(E')\otimes_{\mathbb{F}_q}k$ is an elementary subalgebra of rank $R'$, contradicting the maximality of $R$.
\end{proof}

\begin{remark}
In the proof of Corollary \ref{MaximalElemAb}, one could also construct the group $\{g^{\lambda}\mid g\in E,\;\lambda\in\mathbb{F}_q\}$ as an $\mathbb{F}_q$-linear elementary abelian subgroup containing $E$.  In fact, we have the equality $\sigma(\langle\sigma^{-1}(E)\rangle)=\{g^{\lambda}\mid g\in E,\;\lambda\in\mathbb{F}_q\}$.  This motivates the notation $\langle E\rangle_{\mathbb{F}_q}$ for the group $\{g^{\lambda}\mid g\in E,\;\lambda\in\mathbb{F}_q\}$, which is the smallest $\mathbb{F}_q$-linear subgroup containing $E$.
\end{remark}

Corollary \ref{MaximalElemAb} allows us to relate the maximal rank of an elementary abelian $p$-subgroup in $G(\mathbb{F}_q)$, known as the $p$-rank of $G(\mathbb{F}_q)$, with $R(\mathrm{Lie}(G))$.  The $p$-ranks of the finite simple groups of Lie type are known (cf. Table 3.3.1 in \cite{GLS}).  This leads to Table $\ref{Largest}$, which presents $R(\mathfrak{g})$ for the simple Lie algebras.

{\renewcommand{\arraystretch}{2}

\begin{table}[ht]
\caption{Dimension of largest elementary subalgebra of simple $\mathfrak{g}$}\label{Largest}
\begin{minipage}[b]{0.5\linewidth}
\hfill
\begin{tabular}{|c|c|}
\hline
Type&$R(\mathfrak{g})$\\
\hline
$A_n$&$\left\lfloor\left(\frac{n+1}{2}\right)^2\right\rfloor$\\
\hline
$B_n$, $n=2,3$&$(2n-1)$\\
\hline
$B_n$, $n\ge 4$&$1+\binom{n}{2}$\\
\hline
$C_n$, $n\ge 3$&$\binom{n+1}{2}$\\
\hline
$D_n$, $n\ge 4$&$\binom{n}{2}$\\
\hline
\end{tabular}\hspace*{.1in}\\
\end{minipage}
\begin{minipage}[b]{0.5\linewidth}
\hspace{.05in}
\begin{tabular}{|c|c|}
\hline
$\quad\;\;$Type$\quad\;\;$&$\;\;\;\;R(\mathfrak{g})\;\;\;\;$\\
\hline
$E_6$&$16$\\
\hline
$E_7$&$27$\\
\hline
$E_8$&$36$\\
\hline
$F_4$&$9$\\
\hline
$G_2$&$3$\\
\hline
\end{tabular}\\
\end{minipage}
\end{table}
}

\begin{example}
Let $G=\mathrm{SO}_3$, $p\ge 3$, and $r=1$.  Then $\mathrm{Lie}(SO_3)=\mathfrak{so}_3$ is the collection of skew-symmetric $3\times 3$ matrices.  A skew-symmetric $3\times 3$ nilpotent matrix has the form:
\[
\begin{pmatrix}0&x&y\\-x&0&z\\-y&-z&0\end{pmatrix}
\]
where $x^2+y^2+z^2=0$.  It follows that $\mathbb{E}(1,\mathfrak{so}_3)$ is the irreducible projective variety in $\mathbb{P}^2$ of all points $[x:y:z]$ satisfying $x^2+y^2+z^2=0$.  This equation has $q^2$ solutions over $\mathbb{F}_q$ (exercise!), one of which is $(0,0,0)$.  This leaves us with a set $S$ of $q^2-1$ non-trivial solutions, each of which spans a one dimensional $\mathbb{F}_q$-expressible subalgebra.  Each $\mathbb{F}_q$-expressible subalgebra contains exactly $q-1$ elements of $S$, so that there are $q+1$ different $\mathbb{F}_q$-expressible subalgebras.  A quick computation in Magma shows it is also true that there are $q+1$ subgroups of the form $(\mathbb{Z}/p\mathbb{Z})^d$ in $SO_3(\mathbb{F}_q)$, for many small values of $q$ ($q<400$) as we expect from Corollary \ref{Bij}.  Notice here that $d$ is the maximal rank of an elementary abelian subgroup in $\mathrm{SO}_3(\mathbb{F}_q)$, so that all such subgroups are $\mathbb{F}_q$-linear by Corollary \ref{MaximalElemAb}.
\end{example}

\begin{remark}
In this section we have chosen to use Quillen's category of elementary abelian subgroups as a motivation for defining our category of elementary subalgebras.  Our reasoning behind this is due to the importance of Quillen's category in the cohomology of the group $G(\mathbb{F}_p)$.  It is the author's hope that the isomorphic category $\mathcal{C}_{\mathfrak{g}}(\mathbb{F}_p)$,  or the larger category $\mathcal{C}_{\mathfrak{g}}$ might hold a similar importance in the cohomology of $\mathfrak{g}$.

Another approach for this section would be to motivate our definitions by the Quillen complex of elementary abelian subgroups, that is, the complex associated to the poset of elementary abelian subgroups of $G$ ordered by inclusion.  With a similar definition of the complex of elementary subalgebras, we find that Quillen's complex for $G(\mathbb{F}_p)$ is isomorphic to the subcomplex of $\mathbb{F}_p$-expressible subalgebras.  Using the appropriate analogues of group-theoretic notions (for example, the Frattini subalgebra of $\mathfrak{g}$ as studied in \cite{LT}), much of the machinery developed in Part $2$ of \cite{Smith} for groups can be developed for Lie algebras.  In particular, it is true that the subcomplex of $p$-subalgebras is $G$-homotopy equivalent to the subcomplex of elementary subalgebras, and the proof follows that of the analogous statement for groups (see (\cite{Smith}, Theorem 4.2.4) for a proof using Quillen's Fiber Theorem).  These connections further emphasize that in the study of restricted Lie algebra cohomology, elementary subalgebras of restricted Lie algebras are the appropriate analogue to elementary abelian subgroups of finite groups.
\end{remark}

\section{$G$-orbits of $\mathbb{E}(r,\mathfrak{g})$ defined over $\mathbb{F}_q$}\label{G-Orb}

In this section we use Lang's theorem to show that $\mathbb{F}_q$-rational points exist in the $G$-orbits of $\mathbb{E}(r,\mathfrak{g})$ that are defined over $\mathbb{F}_q$.  By the previous sections, these points correspond exactly to $\mathbb{F}_q$-linear elementary abelian $p$-groups of rank $rd$ in $G(\mathbb{F}_q)$.  This leads to Theorem \ref{FinOrb}, which gives a bijection between the $G$-orbits of $\mathbb{E}(r,\mathfrak{g})$ defined over $\mathbb{F}_q$ and $G$-conjugacy classes of $\mathbb{F}_q$-linear elementary abelian $p$-subgroups of rank $rd$ in $G(\mathbb{F}_q)$.  We clarify the phrase "$G$-conjugacy classes of $\mathbb{F}_q$-linear elementary abelian $p$-groups of rank $rd$ in $G(\mathbb{F}_q)$."  Two elementary abelian $p$-subgroups of rank $r$ in $G(\mathbb{F}_q)$, say $E_r$ and $E'_r$, are $G$-conjugate if there is $g\in G(k)$ such that $g$ conjugates $E_r$ to $E'_r$.  Notice this is not the same as the standard notion of conjugate subgroups in $G(\mathbb{F}_q)$, as there may be non-conjugate subgroups $H,K\subset G(\mathbb{F}_q)$ which are conjugate when viewed as subgroups in $G(\mathbb{F}_{q^e})$.  Also notice that by equation \eqref{GenExp}, the conjugate of an $\mathbb{F}_q$-linear subgroup is $\mathbb{F}_q$-linear, so that the notion of a conjugacy class of $\mathbb{F}_q$-linear subgroups is well-defined.

E. Friedlander has asked for sufficient conditions such that $\mathbb{E}(r,\mathfrak{g})$ is a finite union of $G$-orbits.  Theorem \ref{FinOrb} shows that if $G$ is connected and reductive, and if $p$ satisfies condition $(\star)$ then $\mathbb{E}(r,\mathfrak{g})$ has finitely many $G$-orbits defined over $\mathbb{F}_q$.  This of course does not resolve the question, but as Example \ref{InfOrb} shows, any list of sufficient conditions is sure to be fairly restrictive.  E. Friedlander has conjectured that if $R=R(\mathfrak{g})$ is the largest integer such that $\mathbb{E}(R,\mathfrak{g})$ is non-empty, then $\mathbb{E}(R,\mathfrak{g})$ is a finite union of $G$-orbits.  Proposition \ref{FriedConj} reduces this conjecture to showing that the number of conjugacy classes of elementary abelian subgroups of rank $Rd$ is bounded as $d$ grows.  This is known to be true for all simple groups of classical type, so we obtain an affirmative answer to the conjecture for a large class of groups, namely those $G$ whose derived subgroup is an almost-direct product of simple groups of classical type.

Theorem \ref{FinOrb} also provides a method for bounding the number of $G$-orbits defined over $\mathbb{F}_q$.  In section \ref{Group}, we use Magma to bound the number of $\mathrm{GL}_n$-orbits defined over $\mathbb{F}_p$ in $\mathbb{E}(r,\mathfrak{gl}_n)$ for $n\le 5$.

As mentioned in the introductory paragraph to this section, Theorem \ref{FinOrb} follows from the previous sections and the following theorem of Lang.

\begin{theorem}[\cite{Lang}, Theorem 2]\label{Lang}
Let $G$ be an algebraic group defined over a finite field $F$, and let $V$ be a variety defined over $F$ on which $G$ acts morphically and transitively.  Then $V$ has an $F$-rational point.
\end{theorem}

We should clarify that in Theorem \ref{Lang}, an action is transitive if there is a $v\in V$ such that $V=G\cdot v$.  We do not require the map $G/G_{v}\to V$ to be an isomorphism of varieties.  The proof of the following lemma is immediate from Theorem \ref{Lang}.

\begin{lemma}\label{RatPts}
Let $r$ be such that $\mathbb{E}(r,\mathfrak{g})$ is nonempty, and let $\mathcal{O}$ be any $G$-orbit of $\mathbb{E}(r,\mathfrak{g})$.  If $\mathcal{O}$ is defined over $\mathbb{F}_q$, then the set of $\mathbb{F}_q$-rational points of $\mathcal{O}$ is non-empty.
\end{lemma}

\begin{theorem}\label{FinOrb}
Let $G$ be connected and reductive, and let $p$ satisfy condition $(\star)$.  The $G$-orbits of $\mathbb{E}(r,\mathfrak{g})$ defined over $\mathbb{F}_q$ are in bijection with the $G$-conjugacy classes of $\mathbb{F}_q$-linear elementary abelian $p$-groups of rank $rd$ in $G(\mathbb{F}_q)$.  In particular, $\mathbb{E}(r,\mathfrak{g})$ contains finitely many $G$-orbits defined over $\mathbb{F}_q$.  Furthermore, the number of $\mathbb{F}_q$-rational points of a $G$-orbit defined over $\mathbb{F}_q$ is equal to the size of its corresponding $G$-conjugacy class.
\end{theorem}

\begin{proof}
Let $\mathcal{O}$ be any $G$-orbit of $\mathbb{E}(r,\mathfrak{g})$ defined over $\mathbb{F}_q$.  By Lemma \ref{RatPts}, $\mathcal{O}(\mathbb{F}_q)$ is non-empty.  Furthermore, by Corollary \ref{Bij}, $\mathcal{O}(\mathbb{F}_q)$ is in bijective correspondence with a collection of $\mathbb{F}_q$-linear elementary abelian $p$-subgroups of rank $rd$ in $G(\mathbb{F}_q)$.  By $G$-equivariance, these elementary abelian $p$-subgroups form a $G$-conjugacy class.  Conversely, starting with a $G$-conjugacy class of $\mathbb{F}_q$-linear elementary abelian $p$-subgroups of rank $rd$, Corollary \ref{Bij} gives us a $G$-orbit of $\mathbb{E}(r,\mathfrak{g})$ whose $\mathbb{F}_q$-rational points correspond to elements of the given $G$-conjugacy class.
\end{proof}

\begin{proposition}\label{FriedConj}
Fix $r$, and let $N_d$ be the number of conjugacy classes of elementary abelian $p$-groups of rank $rd$ in $G(\mathbb{F}_q)$.  If $N_d$ is bounded, then $\mathbb{E}(r,\mathrm{Lie}(G))$ is a finite union of $G$-orbits.  In particular, $N_d$ is bounded for simple groups of classical type when $r=R$, so $\mathbb{E}(R(\operatorname{Lie}(G)),\operatorname{Lie}(G))$ is a finite union of $G$-orbits whenever $(G,G)$ is an almost-direct product of simple groups of classical type.
\end{proposition}

\begin{proof}
Suppose to the contrary that $\mathbb{E}(r,\mathrm{Lie}(G))$ is an infinite union of $G$-orbits.  Then the number of $G$-orbits defined over $\mathbb{F}_q$ approaches infinity as $d$ gets large.  By Theorem \ref{FinOrb}, the number of $G$-orbits defined over $\mathbb{F}_q$ is at most $N_d$, which is bounded, providing a contradiction.

If $G$ is a simple group of classical type, and $r=R(\mathfrak{g})$, it is actually the case that $N_d$ is a constant sequence (see \cite{Barry}), so that $\mathbb{E}(r,\mathfrak{g})$ is a finite union of $G$-orbits.  Now, suppose $G$ is a direct product of simple groups of classical type $G_1\times\ldots\times G_m$.  By Proposition $1.12$ in \cite{CFP}, we have an isomorphism of projective $G$-varieties
\[
\prod_{i=1}^m\mathbb{E}(R_i,\mathfrak{g}_i)\to \mathbb{E}\left(\sum_{i=1}^mR_i,\bigoplus_{i=1}^m\mathfrak{g}_i\right)
\]
where $\mathfrak{g}_i=\operatorname{Lie}(G_i)$ and $R_i=R(\mathfrak{g}_i)$.  This isomorphism is given by sending the $m$-tuple $(\epsilon_1,\ldots,\epsilon_m)$ to $\epsilon_1\oplus\ldots\oplus\epsilon_m$.  Now $G=G_1\times\ldots\times G_m$ acts componentwise, and the action of each $G_i$ has finitely many orbits, so the same is true of the action of $G$.

If $G$ is an almost-direct product of simple groups of classical type, then the isogeny $G_1\times\ldots\times G_m\twoheadrightarrow G$ is separable by condition $(\star)$, so we have $\mathfrak{g}=\operatorname{Lie}(G)=\mathfrak{g}_1\oplus\ldots\oplus\mathfrak{g}_m$, and the orbits of the action of $G$ on $\mathbb{E}(r,\mathfrak{g})$ coincide with those of the action of $G_1\times\ldots\times G_m$ on $\mathbb{E}(r,\mathfrak{g})$.

Finally, if $G$ is connected and reductive, then since $G=Z(G)\cdot (G,G)$, the $G$-orbits of $\mathbb{E}(r,\mathfrak{g})$ coincide with the $(G,G)$-orbits.
\end{proof}

\begin{remark}
We also expect, but are unable to verify that $N_d$ is constant for the exceptional simple groups.  If this were true, then $\mathbb{E}(R(\mathrm{Lie}(G)),\mathrm{Lie}(G))$ is a finite union of $G$-orbits for all connected, reductive groups $G$.
\end{remark}

\begin{example}
For an example illustrating that our reductive hypothesis may be unnecessary, consider the non-reductive group $U_3$, the unipotent radical of $B_3$, the group of upper triangular $3\times 3$ matrices.  Example $1.7$ in \cite{CFP} shows that $\mathbb{E}(2,\mathfrak{u}_3)\cong\mathbb{P}^1$.  Explicitly, any elementary subalgebra has a basis of the form:

\begin{equation}\label{Basis}\left\{\begin{pmatrix}0&0&1\\0&0&0\\0&0&0\end{pmatrix},\begin{pmatrix}0&a&0\\0&0&b\\0&0&0\end{pmatrix}\right\}\end{equation}
and this basis is unique up to scalar multiple of the vector $(a,b)$.  A computation shows that each such subalgebra is fixed under conjugation by $U_3$, so that the $G$-variety $\mathbb{E}(2,\mathfrak{u}_3)\cong\mathbb{P}^1$ has infinitely many $G$-orbits (each point is an orbit).  However, only the $\mathbb{F}_q$-rational points of $\mathbb{P}^1$, of which there are finitely many, are orbits defined over $\mathbb{F}_q$.

\end{example}

\begin{example}\label{GL32}
Let $G=\mathrm{GL}_3$, $p\ge 3$, and $r=2$.  Any elementary subalgebra $\epsilon\in\mathbb{E}(2,\mathfrak{gl}_3)$ can be put in upper-triangular form, so $\epsilon$ is conjugate to a subalgebra $\epsilon'$ with basis given by \eqref{Basis} for some $[a:b]\in\mathbb{P}^1$.  In general, the element $[a:b]$ is not defined by $\epsilon$, as conjugating $\epsilon '$ by
\[
\begin{pmatrix}
\lambda&0&0\\0&1&0\\0&0&1
\end{pmatrix}
,\;\lambda\ne 0
\]
gives an upper triangular subalgebra corresponding to $[\lambda a:b]$.  It follows that all subalgebras $\epsilon\in\mathbb{E}(2,\mathfrak{gl}_3)$ that have an element of rank $2$ are conjugate.  The dimension of this orbit is shown to be $4$ in Example $3.20$ of \cite{CFP}.

The only other subalgebras are those whose non-zero elements all have rank equal to $1$.  These subalgebras are conjugate to upper-triangular subalgebras corresponding to the points $[1:0]$ and $[0:1]$.  These two upper-triangular subalgebras are not conjugate (in short, the conditions for a matrix $A$ to conjugate $[1:0]$ to $[0:1]$ require that $\det(A)=0$, a contradiction).  The dimension of each of these two distinct orbits is $2$ (Example 3.20, \cite{CFP}).  We have thus verified that $\mathbb{E}(2,\mathfrak{gl}_3)$ is the union of three $\mathrm{GL}_3$-orbits, all of which are defined over $\mathbb{F}_p$.  As expected from Theorem \ref{FinOrb}, any elementary abelian $p$-subgroup of rank $2d$ in $\mathrm{GL}_3(\mathbb{F}_q)$ is conjugate to exactly one of the groups $\langle I+\lambda E_{12}+\lambda E_{23},I+\mu E_{13}\rangle$, $\langle I+\lambda E_{23},I+\mu E_{13}\rangle$, and $\langle I+\lambda E_{12},I+\mu E_{13}\rangle$.  Here $E_{ij}$ is the matrix whose only non-zero entry is a $1$ in the $i$th row and $j$th column.  This example will be further developed in Proposition \ref{A23} below.
\end{example}

\begin{example}\label{InfOrb}
The following example (due to R. Guralnick) shows that even if $G$ is connected and reductive, $\mathbb{E}(r,\mathfrak{g})$ may be an infinite union of orbits.  Let $G=\mathrm{GL}_{2n}$ and let $\epsilon$ be the elementary subalgebra of $\mathfrak{g}=\mathfrak{gl}_{2n}$ of dimension $n^2$ whose matrices only have nonzero entries in the upper-right $n\times n$ block.  For any $r\le n^2$, we have $\mathrm{Grass}(r,\epsilon)\subset\mathbb{E}(r,\mathfrak{g})$ so that $\dim(\mathbb{E}(r,\mathfrak{g}))\ge\dim(\mathrm{Grass}(r,\epsilon))=(n^2-r)r.$  If $r$ and $n^2$ are such that $(n^2-r)r>4n^2$, then $\dim(\mathbb{E}(r,\mathfrak{g}))>\dim(G)$, so that $\mathbb{E}(r,\mathfrak{g})$ is not a finite union of $G$-orbits.
\end{example}

\begin{question}\label{GroupCond}
As with nilpotent orbits of $\mathfrak{g}$, we can place a partial ordering on the $G$-orbits of $\mathbb{E}(r,\mathfrak{g})$ by $\mathcal{O}\le\mathcal{O}'$ if and only if $\overline{\mathcal{O}}\subset\overline{\mathcal{O}'}$.  For classical Lie algebras, the ordering on nilpotent orbits $(r=1)$ corresponds to the dominant ordering on Jordan type.  For $r>1$, given two $G$-conjugacy classes $C$ and $C'$ of $\mathbb{F}_q$-linear elementary abelian $p$-groups of rank $rd$ in $G(\mathbb{F}_q)$ with corresponding orbits $\mathcal{O}$ and $\mathcal{O}'$ defined over $\mathbb{F}_q$, is there some group-theoretic condition on $C$ and $C'$ that determines when $\mathcal{O}\le\mathcal{O}'$?  In other words, can we describe the partial ordering on orbits in the group setting?  Notice that the existence of a unique maximal element in the partial order implies that $\mathbb{E}(r,\mathfrak{g})$ is irreducible.  Describing the partial order in the group setting might allow us to find further examples of groups $G$ such that $\mathbb{E}(r,\mathrm{Lie}(G))$ is irreducible.
\end{question}

\begin{example}\label{JordanType}
For the case $G=\mathrm{GL}_n$ and $r=1$, the answer to Question \ref{GroupCond} is already known.  For each unipotent $g\in \mathrm{GL}_n(\mathbb{F}_q)$, we know by the theory of Jordan forms that $g$ is conjugate to a direct sum of Jordan blocks, all with eigenvalue $1$.  As in the nilpotent case, the orbits are ordered by Jordan type of the corresponding unipotent elements.  This result is expected, since Springer has shown that $\mathcal{U}(\mathrm{GL}_n)$ and $\mathcal{N}(\mathfrak{gl}_n)$ are isomorphic as $\mathrm{GL}_n$-varieties.
\end{example}

\begin{remark}
Example \ref{JordanType} suggest that the answer to Question \ref{GroupCond} for classical Lie algebras may lie in the Jordan types of the elements of $C$ and $C'$.  That is, there may be a condition on the Jordan types of elements in $C$ and $C'$ that determines when $\mathcal{O}\le\mathcal{O}'$.
\end{remark}

\section{Irreducibility of $\mathbb{E}(r,\mathfrak{g})$}\label{IrredSec}

E. Friedlander has asked for sufficient conditions such that $\mathbb{E}(r,\mathfrak{g})$ is irreducible.  Here we use known results on the irreducibility of the commuting variety of nilpotent matrices to deduce irreducibility for $\mathbb{E}(r,\mathfrak{g})$.

Work of A. Premet in \cite{Premet} shows that $\mathbb{E}(2,\mathfrak{gl}_n)$ is irreducible for all $n$.  This is observed in Example 1.6 of \cite{CFP}.  Premet shows that under less restrictive hypotheses on $\mathfrak{g}$ and $p$ than we consider here, the variety of pairs of commuting nilpotent elements in $\mathfrak{g}$ is equidimensional, having irreducible components which are in one-to-one correspondence with distinguished nilpotent orbits.  For $\mathfrak{g}=\mathfrak{gl}_n$, there is only one distinguished nilpotent orbit, namely, the regular orbit.  It follows that $\mathcal{C}_2(\mathcal{N}(\mathfrak{gl}_n))$ is irreducible.  Since open sets of irreducible sets are themselves irreducible, and continuous images of irreducible sets are irreducible, the map of algebraic varieties $\mathcal{C}_2(\mathcal{N}(\mathfrak{gl}_n))^{\circ}\twoheadrightarrow\mathbb{E}(2,\mathfrak{gl}_n)$ discussed at the end of \S \ref{Review} shows that $\mathbb{E}(2,\mathfrak{gl}_n)$ is irreducible.  This same argument shows that $\mathbb{E}(1,\mathfrak{g})$ is irreducible for all $\mathfrak{g}$, as the restricted nullcone $\mathcal{N}(\mathfrak{g})$ is irreducible.

It is also known that $\mathcal{C}_r(\mathcal{N}(\mathfrak{gl}_n))$ is irreducible for $r=3$ and $n\le 6$, so by similar reasoning, it follows that $\mathbb{E}(r,\mathfrak{gl}_n)$ is irreducible for the corresponding pairs $(r,n)$.  We should note that we have proven the implication
\[
\mathcal{C}_r(\mathcal{N}(\mathfrak{g}))\text{ irreducible}\Longrightarrow\mathbb{E}(r,\mathfrak{g})\text{ irreducible}
\]
however, the converse is not true.  In Corollary 4 of (\cite{Sivic}) it is shown that $\mathcal{C}_r(\mathcal{N}(\mathfrak{gl}_n))$ is reducible for all $r\ge 4$ and $n\ge 4$, but Theorem 2.9 in \cite{CFP} shows that $\mathbb{E}(n^2,\mathfrak{gl}_{2n})$ is irreducible for all $n$.  We summarize the above discussion in the following theorem.

\begin{theorem}\label{Irreducible}
The variety $\mathbb{E}(r,\mathfrak{gl}_n)$ is irreducible for the following ordered pairs $(r,n)$: $(1,n)$ for any $n$, $(2,n)$ for any $n$, $(3,n)$ for $n\le 6$, and $(n^2,2n)$ for any $n$.
\end{theorem}

\begin{question}
The reducibility of the variety of $r$-tuples of pairwise-commuting matrices $\mathcal{C}_r(\mathfrak{gl}_n)$ and the variety of $r$-tuples of pairwise-commuting nilpotent matrices $\mathcal{C}_r(\mathcal{N}(\mathfrak{gl}_n))$ has been extensively studied.  As we've already observed, since $\mathcal{C}_r(\mathcal{N}(\mathfrak{gl}_n))^{\circ}$ is open in $\mathcal{C}_r(\mathcal{N}(\mathfrak{gl}_n))$, the irreducibility of the latter implies that of the former.  Are there counterexamples to the converse?  Also, as in the case of $\mathcal{C}_r(\mathcal{N}(\mathfrak{gl}_n))$, is $\mathcal{C}_r(\mathcal{N}(\mathfrak{gl}_n))^{\circ}$ reducible for large enough $r$ and $n$?
\end{question}

\begin{question}\label{RegOrbQuest}
If $X\in\mathfrak{gl}_n$ is a regular nilpotent element and $\epsilon$ is the $n-1$-plane with basis given by $\{X,X^2,\ldots,X^{n-1}\}$, is the orbit of $X$ dense in $\mathbb{E}(n-1,\mathfrak{gl}_n)$?  Proposition 3.19 and Example 3.20 of \cite{CFP} show that it is open in general and dense in the case $n=3$.  We have shown above that the question also has an affirmative answer for the cases $n=4$.  If the orbit of $X$ is indeed dense, then $\mathbb{E}(n-1,\mathfrak{gl}_n)$ is irreducible for all $n\ge 1$.
\end{question}

\section{Computations for $\mathrm{GL}_n$}\label{Group}

Since the $G$-orbits of $\mathbb{E}(r,\mathfrak{g})$ defined over $\mathbb{F}_q$ are in bijection with the $G$-conjugacy classes of $\mathbb{F}_q$-linear elementary abelian $p$-groups of rank $rd$ in $G(\mathbb{F}_q)$, we can bound the number of such $G$-orbits by computing in the finite group $G(\mathbb{F}_q)$.  In this section we make some computations for $G=\mathrm{GL}_n$, and $d=1$, using the ``ElementaryAbelianSubgroups" function in Magma.

The values appearing in Tables \ref{NumOrb} and \ref{DimOrb} below, as well as the computation of Example \ref{34}, rely on Conjecture \ref{OrbitSize}.  Before stating the conjecture, we introduce some notation.  Any reductive algebraic group $G$ has a Chevalley $\mathbb{Z}$-form, denoted $G_{\mathbb{Z}}$.  Let $G_p=G_{\mathbb{Z}}\times_{\mathbb{Z}}\mathrm{Spec}\;\overline{\mathbb{F}_p}$.  Also, if $\varphi:G\to G'$ is a map of reductive groups defined over $\mathbb{Z}$, then the differential $d\varphi$ induces a map from the $G_p$-orbits of $\mathbb{E}(r,\mathrm{Lie}(G_p))$ to the $G'_p$-orbits of $\mathbb{E}(r,\mathrm{Lie}(G'_p))$.  This follows because $d\varphi\circ\mathrm{Ad}_g=\mathrm{Ad}_{\varphi(g)}\circ d\varphi$.  By abuse of notation, we will also denote the induced map on orbits by $d\varphi$.

\begin{conjecture}\label{OrbitSize}
\leavevmode
\begin{enumerate}
\item For every reductive algebraic group $G$ and for any pair of primes $p,p'$ satisfying condition $(\star)$, there is a natural dimension-preserving bijection $f_{p,p'}$ between the $G_p$-orbits of $\mathbb{E}(r,\mathrm{Lie}(G_p))$ defined over $\mathbb{F}_{p^d}$ and the $G_{p'}$-orbits of $\mathbb{E}(r,\mathrm{Lie}(G_{p'}))$ defined over $\mathbb{F}_{p'^d}$.  By natural, we mean that for any primes $p$, $p'$, and $p''$ satisfying condition $(\star)$ and any map $\varphi:G\to G'$ of reductive groups defined over $\mathbb{Z}$ we have:
\begin{enumerate}
\item $f_{p,p''}=f_{p',p''}\circ f_{p,p'}$ and
\item $d\varphi\circ f_{p,p'}=f'_{p,p'}\circ d\varphi$, where $f_{p,p'}$ and $f'_{p,p'}$ are the bijections for $G$ and $G'$ respectively.
\end{enumerate}
\item Fix a prime $p$ satisfying condition $(\star)$ and a $G_p$-orbit of $\mathbb{E}(r,\mathrm{Lie}(G_p))$ of dimension $e$ defined over $\mathbb{F}_{p^d}$, denoted $\mathcal{O}_p$.  Then for all primes $p'$ satisfying condition $(\star)$ the counting function $p'$ to $\#f_{p,p'}(\mathcal{O}_{p})(\mathbb{F}_{p'})$ is a polynomial in $p'$ of degree $ed$.
\end{enumerate}
\end{conjecture}

The usefulness of assuming the conjecture in what follows is that it allows us to compute the dimension of an orbit by finding the degree of the polynomial which counts the orbit's rational points.  Table \ref{NumOrb} below records experimental results for upper bounds on the number of $\mathrm{GL}_n$-orbits of $\mathbb{E}(r,\mathfrak{gl}_n)$ defined over $\mathbb{F}_p$ for varying $r$ and $n$.  To be specific, Table \ref{NumOrb} records (roughly) the number of conjugacy classes of elementary abelian subgroups of rank $r$ in $\mathrm{GL}_n(\mathbb{F}_p)$.  In light of Conjecture \ref{OrbitSize}, the primes used in the computation have been suppressed, although it is important to note that in all cases $p\ge n=h(\mathrm{GL}_n)$.  The numbers recorded are upper bounds because it is not clear to the author how to determine when two $G(\mathbb{F}_p)$-conjugate subgroups merge in some $G(\mathbb{F}_q)$ (in certain cases, it can be inferred from the sizes of the conjugacy classes which classes merge.  It is the detecting of this merging that causes the numbers in Table \ref{NumOrb} to be slightly less than the number of conjugacy classes of elementary abelian subgroups of rank $r$ in $\mathrm{GL}_n(\mathbb{F}_p)$).  By Theorem \ref{Irreducible}, all the varieties represented here are irreducible, except for those corresponding to $(r,n)=(4,5),\;(5,5)$ and $(5,6)$.  $\mathbb{E}(6,\mathfrak{gl}_5)$ is known to be reducible, which follows from the fact that for $n>1$, $\mathbb{E}(n(n+1),\mathfrak{gl}_{2n+1})$ is the disjoint union of two connected components both isomorphic to $\mathrm{Grass}(n,2n+1)$ (\cite{CFP}, Theorem 2.10).  It is not known to the author if $\mathbb{E}(5,\mathfrak{gl}_5)$ or $\mathbb{E}(4,\mathfrak{gl}_5)$ are irreducible.  Table \ref{DimOrb} records the dimension of the largest orbit from Table \ref{NumOrb}.  In all known cases, this dimension is equal to the dimension of $\mathbb{E}(r,\mathfrak{g})$, suggesting the dimension of $\mathbb{E}(r,\mathfrak{g})$ may be determined by its largest orbit defined over $\mathbb{F}_p$.

\begin{table}[ht]
\caption{Upper bounds for $|\{G\text{-orbits of }\mathbb{E}(r,\mathfrak{gl}_n)\text{ defined over }\mathbb{F}_p\}|$}\label{NumOrb}
\centering
\begin{tabular}{c|c c c c c c}
\hline\hline
$n$ &$r=1$ & $r=2$ & $r=3$ & $r=4$&$r=5$&$r=6$\\ [0.5ex]
\hline
2&1&0&0&0&0&0 \\
3&2 &3&0&0&0&0 \\
4 & 4 & 10&8&1&0&0 \\
5 &6& 35&67&32&5&2\\[1ex]
\hline
\end{tabular}\\[1ex]
\end{table}

\begin{table}[ht]
\caption{Dimension of largest orbit of $\mathbb{E}(r,\mathfrak{gl}_n)$ defined over $\mathbb{F}_p$}\label{DimOrb}
\centering
\begin{tabular}{c|c c c c c c}
\hline\hline
$n$ &$r=1$ & $r=2$ & $r=3$ & $r=4$&$r=5$&$r=6$\\ [0.5ex]
\hline
2&1&0&0&0&0&0 \\
3&5&4&0&0&0&0 \\
4 &11&11&9&4&0&0 \\
5 &19&20&19&16&11&6\\[1ex]
\hline
\end{tabular}\\[1ex]
\end{table}

\begin{question}\label{Formula}
Can we expect a formula for the $(r,n)$ entry of either table?
\end{question}

Example \ref{JordanType} shows that the entry in the $n$th row of column $r=1$ in Table \ref{NumOrb} is exact, and is equal to $p(n)-1$, where $p(n)$ is the number of partitions of the integer $n$.  We lose the trivial partition $1+1+\ldots+1=n$ because this corresponds to the trivial subgroup.  For $r\ge 2$, we would need more exact data to determine if there is a closed formula for the number of orbits defined over $\mathbb{F}_p$.  This closed formula may involve $p(n)$.

There is at least hope for an affirmative answer to Question \ref{Formula} for Table \ref{DimOrb}, as evidenced by the following.  Since $\mathbb{E}(1,\mathfrak{gl}_n)$ is the projectivized nullcone, the fact that the entries in column $r=1$ are $n^2-n-1$ follows from the well-known formula $\dim(\mathcal{N}(\mathfrak{g}))=n^2-n$.  Furthermore, in Example $1.6$ of \cite{CFP} the authors use work of Premet in \cite{Premet} to establish that the entries in column $r=2$ are $n^2-5$, which agrees with our computation.  Also in \cite{CFP}, the identifications of $\mathbb{E}(n^2,\mathfrak{gl}_{2n})$ with $\mathrm{Grass}(n,2n)$ and $\mathbb{E}(n(n+1),\mathfrak{gl}_{2n+1})$ with $\mathrm{Grass}(n,2n+1)$ give the following formulas (which agree with Table \ref{DimOrb}):
\[\dim(\mathbb{E}(n^2,\mathfrak{gl}_{2n}))=n(2n-n)=n^2\]
\[\dim(\mathbb{E}(n(n+1),\mathfrak{gl}_{2n+1})=n(2n+1-n)=n(n+1)\]

It is also known that if $\mathcal{C}_r(\mathcal{N}(\mathfrak{gl}_n))$ is irreducible, then it has dimension $(n+r-1)(n-1)$, as reviewed in Proposition $1$ and Corollary $2$ of \cite{Sivic}.  This together with the map $\mathcal{C}_r(\mathcal{N}(\mathfrak{g}))^{\circ}\twoheadrightarrow \mathbb{E}(r,\mathfrak{g})$ discussed at the end of \S \ref{Review}, whose fibers are $\mathrm{GL}_r$-torsors, show that
\begin{equation}\label{Dim}
\dim(\mathbb{E}(r,\mathfrak{gl}_n))=(n+r-1)(n-1)-r^2
\end{equation}
for all ordered pairs $(r,n)$ for which $\mathcal{C}_r(\mathcal{N}(\mathfrak{gl}_n))$ is irreducible (all known such ordered pairs are presented in \S \ref{IrredSec}).  Notice that equation \eqref{Dim} subsumes the results for $r=1$ and $r=2$, and surprisingly it even agrees with Table \ref{DimOrb} in entries $(4,5)$ and $(5,5)$ for which $\mathcal{C}_r(\mathcal{N}(\mathfrak{gl}_n))$ is known to be reducible.  In fact, the only entries of Table \ref{DimOrb} that don't agree with equation \eqref{Dim} are $(3,3)$, $(4,4)$, and $(6,5)$.

The following proposition is another piece of evidence that the entries of Table \ref{DimOrb} may have a closed form.

\begin{proposition}\label{RegOrb}
Let $\mathcal{O}$ be the open $\mathrm{GL}_n$-orbit of $\mathbb{E}(n-1,\mathfrak{gl}_n)$ containing the subalgebra $\epsilon$ spanned by the powers of a regular nilpotent element $X$ (cf. \cite{CFP},Proposition $3.19$).  Then $\dim(\mathcal{O})=(n-1)^2$.
\end{proposition}

\begin{proof}
Let $G=\mathrm{GL}_n$.  We will show the dimension of the stabilizer $G_{\epsilon}$ of $\epsilon$ has dimension $2n-1$, from which we obtain $\dim(\mathcal{O})=\dim(G)-\dim(G_{\epsilon})=n^2-(2n-1)=(n-1)^2$.  We may choose $X$ to be the Jordan block of size $n$ with eigenvalue $0$, so that $G_{\epsilon}$ consists solely of upper-triangular matrices.  In this case, I claim that $G_{\epsilon}$ is isomorphic to the $(2n-1)$-dimensional quasi-affine variety defined by $x_1\ne 0$, $x_2\ne 0$ in $\mathbb{A}^{2n-1}$.  The map defining the isomorphism is given by sending a matrix $A=(a_{ij})$ which normalizes $\mathcal{O}$ to the point $(a_{11},a_{22},a_{12},a_{23},a_{13},a_{24},\ldots,a_{1(n-1)},a_{2n},a_{1n})$.  For injectivity, we must show that the entries in the top two rows of $A$ along with the condition $A\in G_{\epsilon}$ completely determine $A$.  For surjectivity, we must show that any choice of entries in the top two rows of $A$ define a matrix $A\in G_{\epsilon}$ as long as $a_{11}\ne 0$, $a_{22}\ne 0$ and $a_{21}=0$.

What follows is rather tedious, but the basic idea is that the entries along a super-diagonal are determined uniquely by the first two entries in the super-diagonal, and these first two entries may be arbitrary (except in the case of the diagonal, in which case the entries must be non-zero).  By a super-diagonal, we mean any collection of entries of the form $a_{i,j}$ where $j-i=k$ for some fixed $k=1,\ldots,n-1$.

First, notice that $A\in G_{\epsilon}$ if and only if $AXA^{-1}\in\epsilon$.  For $j>i$, the $(i,j)$ entry of $AXA^{-1}$ is
\begin{equation}\label{ijentry}
\frac{a_{ii}a_{i+1,j}}{a_{i+1,i+1}a_{jj}}-\frac{a_{i,i+1}a_{i+2,j}}{a_{i+2,i+2}a_{jj}}+\ldots-\frac{a_{i,j-2}a_{j-1,j}}{a_{j-1,j-1}a_{jj}}+\frac{a_{i,j-1}}{a_{jj}}
\end{equation}
We've assumed $i+j$ is odd, else we should switch the parity of all the signs.  Notice that $a_{ii}\ne 0$ for all $i$ as $A$ is upper-triangular and invertible.  For fixed $k=1,\ldots,n-1$, the condition that $A\cdot X\cdot A^{-1}\in\epsilon$ requires that \eqref{ijentry} is independent of choice of $(i,j)$ such that $j-i=k$.  For $j-i=1$, \eqref{ijentry} simplifies to $a_{ii}/a_{jj}$.  If this expression is to be independent of $(i,j)$ such that $j-i=1$, once we choose $a_{11}\ne 0$ and $a_{22}\ne 0$, then $a_{ii}$ is determined for $i=3,\ldots,n$.  For $k=2,\ldots,n-2$, if \eqref{ijentry} is independent of $(i,j)$ such that $j-i=k$, then once we choose arbitrary $a_{1,k}$ and $a_{2,k+1}$, then $a_{i,k+i-1}$ is determined for $i=3,\ldots,n-k+1$.  Finally, $a_{1n}$ does not appear in \eqref{ijentry} for any $(i,j)$, so it may be chosen arbitrarily.  This shows that $A$ is uniquely determined by its top two rows (injectivity), and any choice of top two rows $a_{11}\ne 0$, $a_{22}\ne 0$ and $a_{21}=0$ defines a matrix $A\in G_{\epsilon}$ (surjectivity).
\end{proof}

\begin{remark}
A different approach to the computation in Proposition \ref{RegOrb} was shown to the author by E. Friedlander and J. Pevtsova.  They noticed that two regular nilpotent elements $X$ and $Y$ define the same subalgebra if and only if $Y=a_1X+a_2X^2+\ldots+a_{n-1}X^{n-1}$, with $a_1\ne 0$.  These $n-1$ degrees of freedom together with the fact that the regular nilpotent orbit is $n^2-n$ shows that the dimension of $\mathcal{O}$ is $n^2-n-(n-1)=(n-1)^2$.
\end{remark}

If Question \ref{RegOrbQuest} has an affirmative answer, or if any subalgebra defined by a regular nilpotent element is in the orbit of largest dimension, then Proposition \ref{RegOrb} computes the dimension of $\mathbb{E}(n-1,\mathfrak{gl}_n)$ to be $(n-1)^2$, which is verified through $n=5$ in Table \ref{DimOrb}, and agrees with equation \eqref{Dim}.

By considering a smaller subspace defined by a regular nilpotent element, the following proposition gives a lower bound on the dimension of $\mathbb{E}(r,\mathfrak{gl}_n)$ for the intermediate region $1\le r< n-1$.

\begin{proposition}
Let $X$ be a regular nilpotent element in $\mathfrak{gl}_n$ and for $1\le r< n-1$ consider the elementary subalgebra $\epsilon_r=\mathrm{Span}\{X,X^2,\ldots,X^r\}$.  Then $\dim (\mathrm{GL}_n\cdot\epsilon_r)=n^2-n-1$.  In particular, $\dim(\mathbb{E}(r,\mathfrak{gl}_n)\ge n^2-n-1$ for $1\le r< n-1$.
\end{proposition}

\begin{proof}
I claim that the map sending a matrix $A\in G_{\epsilon_r}$ to $(a_{11},a_{22},a_{12},a_{13},\ldots,a_{1n})$ is an isomorphism onto the quasi-affine variety in $\mathbb{A}^{n+1}$ defined by the condition that the first two coordinates are nonzero.  From this claim we have $\dim (\mathrm{GL}_n\cdot\epsilon_r)=n^2-(n+1)$.  The proof of the claim follows similar reasoning of the proof of Proposition \ref{RegOrb}, except that the superdiagonals corresponding to $j-i=k>r$ must now be $0$.  It follows that $g\in G_{\epsilon_r}$ if and only if $g\in N_G(\mathrm{Span}\{X\})$.  To see this, note if $g\in N_G(\epsilon_r)$ and $gXg^{-1}=\sum c_iX^i$ with some $c_i\ne 0$ for $i>1$, then $gX^rg^{-1}$ will have nonzero entries in the superdiagonal corresponding to $k=r+1$.  Hence, for $k>1$ a choice of $a_{1,k}$ determines $a_{i,k+i-1}$ for $i=2,\ldots,n-k+1$.  For $k=1$, we can still choose $a_{1,1}$ and $a_{2,2}$ arbitrarily.  It follows that $\dim{G_{\epsilon_r}}=(n-1)+2=n+1$, so that $\dim (\mathrm{GL}_n\cdot\epsilon_r)=n^2-(n+1)$.  Notice we have shown that the corresponding bound on $\dim(\mathbb{E}(r,\mathfrak{gl}_n)$ is sharp in the limiting case $r=1$.  However, Table \ref{DimOrb} shows that we may have strict inequality for $r=2,\ldots,n-2$.
\end{proof}

Motivated by Conjecture \ref{OrbitSize}, we include two propositions computing the sizes of the different conjugacy classes found in Table \ref{NumOrb}.

\begin{proposition}
There is one orbit in $\mathbb{E}(1,\mathfrak{gl}_2)$, and the number of $\mathbb{F}_p$-rational points is $p+1$.
\end{proposition}

\begin{proof}
A Sylow $p$-subgroup of $\mathrm{GL}(2,p)$ is isomorphic to $\mathbb{Z}/p\mathbb{Z}$, so the Sylow theorems show there is a unique $G$-conjugacy class.  One such group is represented by the matrices
\[\left\{\begin{pmatrix}1&a\\0&1\end{pmatrix}\middle|a\in\mathbb{F}_p\right\}\]
whose stabilizer under conjugation is the group
\[\left\{\begin{pmatrix}a&b\\0&c\end{pmatrix}\middle|a,c\in\mathbb{F}_p^{\times},\;b\in\mathbb{F}_p\right\}\]
This stabilizer has order $p(p-1)^2$, so that the size of the orbit is
\[\frac{|GL(2,p)|}{p(p-1)^2}=\frac{(p^2-p)(p^2-1)}{p(p-1)^2}=p+1\]
The result then follows from Theorem \ref{FinOrb}.
\end{proof}

\begin{proposition}\label{A23}
There are three $G$-orbits in $\mathbb{E}(2,\mathfrak{gl}_3)$, two with $p^2+p+1$ $\mathbb{F}_p$-rational points, and one with $(p^2+p+1)(p+1)(p-1)$ $\mathbb{F}_p$-rational points.
\end{proposition}

\begin{proof}
We know from Example \ref{GL32} that there are $3$ $G$-orbits of $\mathbb{E}(2,\mathfrak{gl}_3)$.  The orbit consisting of subalgebras with elements of rank $2$ has representative
\[
E=\left\langle\begin{pmatrix}1&0&1\\0&1&0\\0&0&1\end{pmatrix},\begin{pmatrix}1&1&0\\0&1&1\\0&0&1\end{pmatrix}\right\rangle
\]
in $\mathrm{GL}_3(\mathbb{F}_p)$.  The normalizer of $E$ is
\[
N_{\mathrm{GL}_3(\mathbb{F}_p)}(E)=\left\{\begin{pmatrix}a&b&c\\0&d&e\\0&0&f\end{pmatrix}\middle|af=d^2\right\}
\]
which has order $p^3(p-1)^2$.  Orbit-stabilizer then gives the size of the orbit:
\[
\frac{|\mathrm{GL}(3,p)|}{p^3(p-1)^2}=\frac{(p^3-p^2)(p^3-p)(p^3-1)}{p^3(p-1)^2}=(p^2+p+1)(p+1)(p-1)
\]
The result for the large orbit follows from Theorem \ref{FinOrb}.  The proof for the sizes of the other two conjugacy classes is similar, and omitted.
\end{proof}

Notice that the dimensions computed in Example \ref{GL32} and the degree of the polynomials in Proposition \ref{A23} provide evidence for the veracity of Conjecture \ref{OrbitSize}.

\begin{example}\label{34}
If Conjecture \ref{OrbitSize} is true, then a computation with Magma shows that $\mathbb{E}(3,\mathfrak{gl}_4)$ contains $8$ $G$-orbits defined over $\mathbb{F}_p$ of dimensions 3, 3, 6, 7, 7, 7, 8, and 9.  The two orbits of dimension 3 must be closed, and by irreducibility, all orbits of degree less than 9 lie in the boundary of the orbit of dimension 9.  Further understanding of the partial order on the orbits is necessary to determine which intermediate orbits lie in the closure of others.  For example, is the orbit of dimension 6 closed, or is one or more of the 3 dimensional orbits found in its closure?  Can this question be answered by some group theoretic condition on the $G$-conjugacy classes of subgroups corresponding to the orbits of dimensions 3, 3, and 6, per Question \ref{GroupCond}?  

Table \ref{GL34Tab} records how many $\mathbb{F}_p$-rational points lie in each orbit.

\begin{table}[ht]
\caption{Orbits of $\mathbb{E}(3,\mathfrak{gl}_4)$ defined over $\mathbb{F}_p$}\label{GL34Tab}
\centering
\renewcommand\arraystretch{1.5}
\begin{tabular}{c c}
\hline\hline
$\dim(\mathcal{O})$ & $\#\mathcal{O}(\mathbb{F}_p)$\\ [0.5ex]
\hline
3&$(p^2+1)(p+1)$\\
3&$(p^2+1)(p+1)$\\
6&$(p^2+p+1)(p^2+1)(p+1)^2$\\
7&$(p^2+p+1)(p^2+1)(p+1)p(p-1)$\\
7&$(p^2+p+1)(p^2+1)(p+1)^2(p-1)$\\
7&$(p^2+p+1)(p^2+1)(p+1)^2(p-1)$\\
8&$(p^2+p+1)(p^2+1)(p+1)p^2(p-1)$\\
9&$(p^2+p+1)(p^2+1)(p+1)^2p(p-1)^2$\\
\end{tabular}
\end{table}
\end{example}

\section*{Acknowledgements}
We would like to thank Eric Friedlander, whose support and guidance were central throughout the research and writing of this paper.  In particular, Eric's contributions include, but are not limited to, the suggestion to use Springer isomorphisms to study elementary subalgebras and the idea to apply Lang's theorem in the proof of Theorem \ref{FinOrb}.  Paul Sobaje is also due our gratitude for explaining the history and development of the canonical Springer isomorphism, and for answering many other questions during the editing process of this paper.  A helpful conversation with Jim Stark contributed to the results of \S \ref{AnIsoOfCat} and their exposition.  We would also like to thank Robert Guralnick for pointing out an error in the first version of this paper, and for contributing Example \ref{InfOrb}.  Finally, we are very grateful to the referee for comments which helped us clarify the results of this paper and better understand them in relation to the relevant literature.

\end{document}